\documentclass[10pt]{amsart}
\usepackage[active]{srcltx}
\usepackage{epsfig}
\usepackage{mathtools}
\usepackage{yfonts}
\usepackage{turnstile}
\usepackage{extarrows}
\usepackage{amsmath}

\newtheorem{theorem}{Theorem}[section]
\newtheorem{lemma}[theorem]{Lemma}
\newtheorem{proposition}[theorem]{Proposition}
\newtheorem{corollary}[theorem]{Corollary}

\theoremstyle{definition}

\theoremstyle{remark}
\newtheorem{remark}[theorem]{Remark}

\numberwithin{equation}{section}

\renewcommand{\Re}{\operatorname{Re}}
\renewcommand{\Im}{\operatorname{Im}}

\def\l{{\lambda}}

\def\R{\mathbb{R}}
\def\CC{\mathbb{C}}

\def\N{\mathbb{N}}

\def\Q {\mathbb{Q}}

\def\ud{\textrm{ d}}

 \makeatletter \newcommand*{\Relbarfill@}{\arrowfill@\Relbar\Relbar\Relbar} \newcommand*{\xeq}[2][]{\ext@arrow 0055\Relbarfill@{#1}{#2}} \makeatother

\def\C{{\mathbb C}}
\def\cD{{\mathcal D}}

\def\cL{{\mathcal L}}

\def\cR{{\mathcal R}}

\def\T{{T}}
\def\Q{{Q}}
\def\L{{L}}
\def\S{{S}}
\def\H{{H}}

\def\V{{V}}
\def\cC{{C}}
\def\a{\mathfrak{a}}
\def\b{\mathfrak{b}}

\def\s{\mathfrak{s}}
\def\t{\mathfrak{t}}
\def\B{{B}}

\newcommand{\dual}[1]{#1^\prime}

\newcommand{\cdual}[1]{{#1}{'}}
\newcommand{\bracket}[2]{\langle #1 , #2 \rangle}

\title[The Kato property of sectorial forms]{Note on the Kato property of sectorial forms}

\author{Ralph Chill}

\address{R.~Chill, Institut f\"ur Analysis, Fakult\"at Mathematik, Technische Universit\"at Dresden, 01062 Dresden, Germany}
\email{ralph.chill@tu-dresden.de}

\author{Sebastian Kr\'ol}
\address{S. Kr\'ol, Faculty of Mathematics and Computer Science, Adam Mickiewicz University Pozna{\'n}, ul. Uniwersytetu Pozna{\'n}skiego 4, 61-614 Pozna{\'n}, Poland}
\email{sebastian.krol@amu.edu.pl}

\begin{document}

\thanks{The second author was partially supported by the Alexander von Humboldt Foundation and NCN grant  UMO-2017/27/B/ST1/00078}

\thanks{}

\date{\today}


\begin{abstract}
We characterise the Kato property of a sectorial form $\a$, defined on a Hilbert space $\V$,  with respect to a larger Hilbert space $\H$ in terms of two bounded, selfadjoint operators $\T$ and $\Q$ determined by the imaginary part of $\a$ and the embedding of $\V$ into $\H$, respectively. As a consequence, we show that if a bounded selfadjoint operator $\T$ on a Hilbert space $\V$ is in the Schatten class $S_p(\V)$ ($p\geq 1$), then the associated form $\a_\T(\cdot,  \cdot) := \bracket{(I+i\T)\cdot}{\cdot}_\V$ has the Kato property with respect to every Hilbert space $\H$ into which $\V$ is densely and continuously embedded. This result is in a sense sharp. Another result says that if $\T$ and $\Q$ commute then the  form $\a$ with respect to $\H$ possesses the Kato property. 
\end{abstract}

\renewcommand{\subjclassname}{\textup{2010} Mathematics Subject Classification}

\maketitle


\section{Introduction and preliminaries}

Let $\a : \V\times\V\to\CC$ be a bounded, sectorial, coercive, sesquilinear form on a complex Hilbert space $\V$, which is densely and continuously embedded into a second Hilbert space $\H$. Then $\a$ induces a sectorial,  invertible operator $\L_\H$ on $\H$, and Kato's square root problem is to know whether the domain of $\L_\H^{\frac12}$ is equal to the form domain $\V$. If this is the case, then we say that the couple $(\a ,\H )$ has the Kato property. In this short note we characterise the Kato property of $(\a ,\H )$ in terms of two bounded, selfadjoint operators $\T$, $\Q\in\cL (\V )$ determined by the imaginary part of $\a$ and by the embedding of $\V$ into $\H$, respectively. We show that the Kato property of $(\a, \H)$ is equivalent to the similarity of $\Q(I+i\T)^{-1}$ to an accretive operator, or to the similarity of $(I+\Q+i\T)(I-\Q+i\T)^{-1}$ to a contraction; see Theorem \ref{charact0}. The established link to different characterisations known in the literature provides an interesting connection between a variety of techniques and results mainly from operator theory of bounded operators, harmonic analysis, interpolation theory, or abstract evolution equations. 

In particular, we show that if a bounded, selfadjoint operator $\T$ on a Hilbert space $\V$ is in the Schatten class $S_p(\V)$ for some $p\geq 1$, then the associated form $\a_\T(\cdot, \cdot):=\bracket{(I+i\T)\cdot}{\cdot}_\V$ has the Kato property with respect to every Hilbert space $\H$ into which $\V$ is densely and continuously embedded; see Corollary \ref{main}. This result is in a sense sharp; see Proposition \ref{counterex}. 

On the other hand, if $\a$ is an arbitrary bounded, sectorial, coercive form on $\V$, then for every nonnegative, injective operator $\Q\in\cL (\V )$ which is of the form $I+P$ with $P\in S_p(\V )$ ($p\geq 1$), the pair $(\a ,\H_\Q)$ has the Kato property, where $\H_\Q$ is the completion of $\V$ with respect to the inner product $\bracket{\Q\cdot}{\cdot}_\V$; see Corollary \ref{cor 3}. Another straightforward consequence of Theorem \ref{charact0} says that for every pair $(\T ,\Q)$ of selfadjoint, commuting operators (with $\Q$ being nonnegative and injective), the form $\a_\T$ has the Kato property with respect to $\H_\Q$; see Corollary \ref{cor 1}.


We conclude this introduction with some preliminaries.

\subsection{Forms}
Let $\a$ be a bounded, sesquilinear form on a complex Hilbert space $\V$. Denote by $\a^*$ the {\em adjoint form} of $\a$, that is, $\a^*(u,v):=\overline{\a(v,u)}$
for every $u,v\in \V$. Then we call
\begin{align*}
\s:= \Re \a & := (\a + \a^*)/2 \quad \text{ and} \\
\t:= \Im \a & := (\a - \a^*)/2i 
\end{align*}
the {\em real part} and the {\em imaginary part} of $\a$, respectively. Note that $\s = \Re\a$ and $\t = \Im\a$ are symmetric forms on $\V$ and $\a=\s+i\t$. Throughout the following, we assume that $\a$ is coercive in the sense that $\Re\a (u,u) \geq \eta \, \| u\|_V^2$ for some $\eta >0$ and every $u\in \V$. This means that $\s=\Re\a$ is an equivalent inner product on $\V$, and for simplicity we assume that $\s$ is equal to the inner product on $\V$: $\s (u,v) = \bracket{u}{v}_\V$ ($u$, $v\in\V$). We shall also assume that $\a$ is {\it sectorial}, that is, there exists $\beta \geq 0$ such that
\begin{equation}\label{sectorial}
|\Im\a (u,u)|\leq \beta \, \Re\a (u,u), \quad u\in \V.
\end{equation}

Let $\H$ be a second Hilbert space such that $\V$ is densely and continuously embedded into $\H$,  that is, there exists a bounded, injective, linear operator $j:\V\to\H$ with dense range. In the sequel we identify $\V$ with $j(\V)$. The embedding $j$ induces a bounded, linear embedding $\cdual{j} : \H \to \dual{\V}$ (where $\dual{\V}$ is the space of bounded, antilinear functionals on $\V$) given by
\[
 \cdual{j} (u) := \bracket{u}{\cdot}_\H , \quad u\in\H . 
\]
Then we have the following picture:
\[
\V \xhookrightarrow{j\ } \H \xhookrightarrow{\cdual{j}} \dual{\V} \text{ and } \V \xhookrightarrow{\cdual{j}j} \dual{\V} .
\]
We write also $J := \cdual{j} j$ for the linear embedding of $\V$ into the dual space $\dual{\V}$. As usual, $\dual{\V}$ is equipped with the inner product $\bracket{u}{v}_{\dual{\V}} := \bracket{I_\V u}{I_\V v}_\V$, where $I_\V:\dual{\V}\rightarrow \V$ is the Riesz isomorphism.

\subsection{Bounded operators associated with the pair $(\a ,\H )$}
Let $(\a , \H)$ be given as above. We define two associated bounded, linear operators on $\V$. In fact, by the Riesz-Fr\'echet representation theorem, there exist two unique selfadjoint operators $\T = \T_\a$, $\Q = \Q_\H \in \cL (\V)$, such that
\begin{align}
\nonumber \t (u,v) & = \bracket{\T u}{v}_\V \text{ and} \\
\label{H and Q} \bracket{u}{v}_\H & = \bracket{\Q u}{v}_\V \text{ for every } u,\, v\in \V , 
\end{align}
and hence, by recalling our convention that $\s = \bracket{\cdot}{\cdot}_\V$,
\begin{equation} \label{a and T}
 \a (u,v) = \bracket{ (I+i\T )u}{v}_\V \text{ for every } u,\, v\in \V.
\end{equation}
Moreover, since $\bracket{\cdot}{\cdot}_\H$ is an inner product, $\Q$ is nonnegative and injective. In fact, $\Q = j^*j$, where $j^* : \H\to\V$ is the Hilbert space adjoint of $j$.  

Conversely, every selfadjoint operator $\T\in\cL (\V )$ induces via the equality \eqref{a and T} a bounded, sesquilinear, sectorial form $\a$ on $\V$ for which $\Re\a$ coincides with the inner product $\bracket{\cdot}{\cdot}_\V$, and for which $\Im\a$ is represented by $\T$. Similarly, every nonnegative, injective operator $\Q \in\cL (\V)$ induces via the equality \eqref{H and Q} an inner product $\bracket{\cdot}{\cdot}_\H := \bracket{\Q\cdot}{\cdot}_\V$ on $\V$, and thus, by taking the completion, a Hilbert space $\H_\Q$ into which $\V$ is densely and continuously embedded. 

We say that the pair of operators $(\T , \Q)$ is associated with the pair $( \a ,\H)$, or, conversely, the pair $(\a ,\H )$ is associated with the pair $(\T , \Q )$. 

\subsection{Unbounded operators associated with the pair $(\a ,\H )$}
Given a pair $(\a ,\H )$ as above, we define also associated closed, linear operators on $\H$ and $\dual{\V}$. 

First, we denote by $\L_\H := \L_{\a ,\H}$ the, in general, unbounded operator on $\H$ given by 
\begin{align*}
 \cD (\L_\H ) & := \{ u\in j(\V) : \exists f\in\H \, \forall v\in \V \, : \, \a (j^{-1}u,v) = \bracket{f}{jv}_\H \} , \\
 \L_\H u & := f .
\end{align*}
Second, we denote by $\L_{\dual{\V}} := \L_{\a ,\dual{\V}}$ the operator on $\dual{\V}$ which is given by 
\begin{align*}
\cD (\L_{\dual{\V}} ) & := (\cdual{j}j) (\V) = J(\V) , \\
\L_{\dual{\V}} u & := \a( J^{-1} u, \cdot) . 
\end{align*}
In a similar way we define the operators $\L_{\s ,\H}$ and $\L_{\s ,\dual \V}$ associated with the real part $\s = \Re\a$.

Recall that a closed, linear operator $(A,\cD (A))$ on a Banach space $X$ is called {\em sectorial of angle $\theta\in (0,\pi )$} if 
\[
 \sigma (A ) \subseteq \Sigma_\theta := \{ z\in\CC : |\arg z| \leq \theta \} ,
\]
and if for every $\theta'\in (\theta ,\pi )$ one has
\[
 \sup_{z\not\in\Sigma_{\theta'}} \| z R(z,A )\| < \infty . 
\]
We simply say that $A$ is {\em sectorial} if it is sectorial for some angle $\theta\in (0,\pi )$. The {\em numerical range} of a closed, linear operator $(A,\cD (A))$ on a Hilbert space $\H$ is the set
\[
 W(A ) := \{ \bracket{A u}{u}_\H : u\in\cD (A), \, \| u\|_\H = 1\} .
\]
The operator $A$ is said to be {\it $\theta$-accretive} for $\theta\in (0,\pi)$, if $W(A) \subseteq \Sigma_\theta$, that is, if
\[
  |\arg\bracket{Au}{u}_\H| \leq \theta \text{ for every } u\in\cD (A).
\]
If $\theta=\frac{\pi}{2}$, that is, $\Re\bracket{Au}{u}_\H \geq 0$ for every $u\in\cD (A)$, we say that $A$ is {\it accretive}.

Both operators $\L_\H$ and $\L_{\dual{\V}}$ defined above are sectorial for some angle $\theta \in (0,\frac{\pi}{2} )$. Since $\a$ is assumed to be coercive, we have $0\in\rho (\L_\H )$ and $0\in\rho (\L_{\dual{\V}} )$, that is, both $\L_\H$ and $\L_{\dual{\V}}$ are isomorphisms from their respective domains onto $\H$ and $\dual{\V}$, respectively; see e.g. \cite[Theorem 2.1, p. 58]{Ya10}.

It is easy to check that the numerical range of $\L_\H$ is contained in the sector $\Sigma_\theta$ with $\theta = \arctan \beta$ and in particular $\L_\H$ is $\theta$-accretive. As a consequence, by \cite[Theorem 11.13]{KuWe04}, $\L_\H$ admits a bounded $H^\infty$ functional calculus. We refer the reader to \cite{KuWe04} or \cite{Hs06} for the background on fractional powers and $H^\infty$ functional calculus of sectorial operators.

\section{Characterisations of the Kato property}
Let $(\a ,\H )$ be as above, that is, $\a$ is a bounded, sectorial, coercive, sesquilinear form on a Hilbert space $\V$ which embeds densely and continuously into a second Hilbert space $\H$. Let $\L_\H = \L_{\a,\H}$ be defined as above. We say that the couple $(\a, \H )$ has {\it Kato's
property} if $\cD(\L_\H^{1/2}) = \V$. By the Closed Graph Theorem, 
if  $(\a, \H )$ has the Kato property, then the norms $\|\L_\H^{1/2} \cdot\|_\H$ and $\|\cdot\|_{\V}$ are equivalent on $\V$. According to Kato \cite{Kat61} and Lions \cite{Lio62}, the coincidence of any two of the spaces $\cD( \L_\H^{1/2})$, $\cD( \L_\H^{*1/2})$ and $\V$  implies the coincidence of the three. 

Moreover, by Subsection 1.2, it is natural to say that a pair $(\T ,\Q )$ of selfadjoint, bounded operators on a Hilbert space $\V$, with $\Q$ being nonnegative and injective, has the Kato property, if the associated pair $(\a_\T, \H_\Q)$ has the Kato property, where $\a_\T (u,v) = \bracket{(I+i\T)u}{v}_\V$ ($u$, $v\in \V$) and $\H_\Q$ is the completion of $(\V , \bracket{\Q\cdot}{\cdot}_\V)$.  

The main result of this section is the following characterisation of the Kato property of $(\a,\H)$ in terms of the associated pair of bounded operators $(\T ,\Q)$.  

\begin{theorem}\label{charact0} 
Let $(\T , \Q)$ be the pair of operators associated with $(\a ,\H)$ as above. Then the following assertions are equivalent:
\begin{itemize}
 \item  [(i)] $(\a, \H )$ has the Kato property
 \item  [(ii)] There exists a positive operator $\S$ on $\V$ such that
\[
\bracket{\Q \S(I+i\T )u}{u}_\V  \in \Sigma_\theta \text{ for every } u\in \V \text{ and some } \theta<\frac{\pi}{2}.
\]
 \item  [(ii')] There exists a positive operator $\S$ on $\V$ such that
\[
\Re\, \bracket{\Q \S(I+i\T )u}{u}_\V \geq 0  \text{ for every } u\in \V.
\]
 \item [(iii)] There exists a positive operator $\S$ on $\V$ such that
\[
\bracket{\S(I-i\T )^{-1} \Q u}{u}_\V \in \Sigma_\theta \text{ for every } u\in \V \text{ and some } \theta<\frac{\pi}{2},
\]
that is, $(I-i\T )^{-1}\Q$ is similar to a $\theta$-accretive operator, or, equivalently, the operator $(I-i\T )^{-1}\Q$ has a bounded $H^\infty(\Sigma_\theta)$ functional calculus. 

\item [(iv)] There exists a positive operator $\S$ on $\V$ such that
\[
 \bracket{\S\Q (I+i\T )^{-1} u}{u}_\V \in \Sigma_\theta \text{ for every } u\in \V \text{ and some } \theta<\frac{\pi}{2},
\]
that is, $\Q (I+i\T )^{-1}$ is similar to a $\theta$-accretive operator, or, equivalently, the operator $\Q(I+i\T )^{-1}$ has a bounded $H^\infty(\Sigma_\theta)$ functional calculus.  

\item [(v)] The operator $(I-\Q+i\T ) (I+\Q +i\T )^{-1}$ is polynomially bounded.

\item [(vi)] The operator $(I-\Q+i\T ) (I+\Q +i\T )^{-1}$ is similar to a contraction. 

\end{itemize}
\end{theorem}

Recall that, if $\T$, $\Q \in\cL (\V )$ are selfadjoint operators, then $\Q \T $ is selfadjoint if and only if $\T$ and $\Q$ commute, or if and only if  $\bracket{\Q \T u}{u}_\V\in \R$ for every $u\in \V$. Therefore, the above Theorem \ref{charact0}(ii') gives the following sufficient condition for $(\a ,\H )$ to have the Kato property.

\begin{corollary}\label{cor 1}
 If $\T$ and $\Q$ commute, then $(\a, \H )$ has the Kato property. 
\end{corollary}

We start with auxiliary results on the operators appearing in Theorem \ref{charact0}.

\begin{lemma}\label{lem on sect} Let $\T$ and $\Q$ be selfadjoint, bounded operators on a Hilbert space $\V$. Assume that $\Q$ is nonnegative. 
Then, the operator $A := \Q (I+i\T )^{-1} \in\cL (\V)$ is sectorial of angle $\theta <\frac{\pi}{2}$. 
\end{lemma}

\begin{proof} 
By a standard argument based on the Neumann series extension it is sufficient to show that $\sup_{\Re z\leq 0}\|zR(z,A)\|<\infty$.
Note that for every $z\in\CC$ with $\Re z\leq 0$ and every $u\in \V$ we have
\begin{align*}
 \bracket{(z+iz \T -\Q )u}{u}_\V & = ( \Re z \, \| u\|_\V^2 - \Im z \, \bracket{\T u}{u}_\V - \bracket{\Q u}{u} ) \\
& \phantom{=} + i\, ( \Im z \, \| u\|_\V^2 + \Re z \, \bracket{\T u}{u}_\V ) , 
\end{align*}
and therefore
\begin{align*}
 | \bracket{(z+iz \T -\Q )u}{u}_\V |^2 & = |z|^2 \, \|u\|_\V^4 - 2\Re z \, \| u\|_\V^2 \bracket{\Q u}{u}_\V + \\
 & \phantom{=} + ( |z|^2 \,\bracket{\T u}{u}^2_\V + 2 \, \Im z \, \bracket{\T u}{u}_\V \, \bracket{\Q u}{u}_\V + \bracket{\Q u}{u}_\V^2 ) \\
& \geq |z|^2 \, \| u\|_\V^4 ,
\end{align*}
or, by the Cauchy-Schwarz inequality,
\[
 \| (z+iz \T -\Q )u\|_\V \geq |z| \, \|u\|_\V .
\]
This inequality implies that $z+iz\T -\Q$ is injective and has closed range. A duality argument, using similar estimates as above, shows that $z+iz\T -\Q$ has dense range, and therefore $z+iz\T -\Q$ is invertible for every $z\in\CC$ with $\Re z\leq 0$. Moreover, the above inequality shows that 
\[
 \sup_{\Re z\leq 0} \| z (z+iz \T -\Q )^{-1} \| \leq 1 .
\]
As a consequence, $z -A = (z+iz \T -\Q ) (I+i\T )^{-1}$ is invertible for every $z\in\CC$ with $\Re z \leq 0$ and 
\[
 \sup_{\Re z\leq 0} \| z R(z,A)\| = \sup_{\Re z\leq 0} \| (I+i\T ) \, z (z+iz \T -\Q )^{-1} \| \leq \|I+i\T \|.
\]
\end{proof}

Let $A\in\cL (\V )$ be a bounded, sectorial operator of angle $\theta\in (0,\frac{\pi}{2} )$, and let $\cC := (I-A)(I+A)^{-1}$ be its {\em Cayley transform}. Then the equality
\begin{align*}
 (z-1) (z-\cC )^{-1} = \frac{z-1}{z+1} \, \left(\frac{z-1}{z+1} +A \right)^{-1} \, (I+A)
\end{align*}
shows that $\cC$ is a {\em Ritt operator}, that is, $\sigma (\cC ) \subseteq {\mathbb D} \cup \{ 1\}$ (where $\mathbb D\subset\CC$ is the open unit disk) and 
\[
\sup_{|z| > 1} \| (z-1)R(z,\cC ) \| <\infty .
\]
From this and the preceding lemma, we obtain the following statement.

\begin{lemma}\label{Ritt cond}
 The Cayley transform 
\[
\cC = (I-A)(I+A)^{-1} = (I-\Q +i\T ) (I+ \Q +i\T )^{-1}
\]
of the operator $A = \Q (I+i\T )^{-1}$ is a Ritt operator. 
\end{lemma}

Recall that a bounded operator $\cC$ on a Hilbert space $\V$ is a Ritt operator if and only if it is power bounded and $\sup_{n\in\N} n\|\cC^n - \cC^{n+1}\|<\infty$; see \cite{NaZe99}. Furthermore, a bounded operator $\cC$ on a Hilbert space is {\em polynomially bounded} if there exists a constant $M\geq 0$ such that for every polynomial $p$ one has
\[
\| p(\cC ) \| \leq M \, \sup_{|z|\leq 1}| p(z)|.
\]

The proof of Theorem \ref{charact0} is a consequence of 
the characterisation of the Kato property by means of the boundedness of the $H^\infty$ functional calculus for the operator $\L_{\dual{\V}} = \L_{\a, \dual{\V}}$ given by Arendt in \cite[Theorem 5.5.2, p.45]{Ar04}. 

\begin{lemma}\label{Arendt charact} 
Let $\L_{\dual{\V}} = \L_{\a ,\dual{\V}}$ be the operator associated with $(\a ,\H )$ as above. Then the following assertions are equivalent:
\begin{itemize}
 \item [(i)] $(\a, \H )$ has the Kato property.
 \item [(ii)] $\L_{\dual{\V}}$ has a bounded $H^\infty$ functional calculus.
\end{itemize}
Moreover, if (i) or (ii) holds, then $\L_{\dual{\V}}$ has a bounded $H^\infty(\Sigma_\theta)$ functional calculus for every $\theta > \arctan \beta$ with $\beta$ as in \eqref{sectorial}.   
\end{lemma}

For the convenience of the reader we recall the proof of this result using our notation and with slight modifications.

\begin{proof}
First of all, note that the operator $\L_\H$ can be expressed as the operator $j'^{-1} \L_{\dual{\V}} j'$ with domain $\{ u\in \H : j'u \in \cD ( \L_{\dual{\V}} \textrm{ and } \L_{\dual{\V}} j'u \in j'(\H)\}$. Then $(\l - \L_\H)^{-1} = j'^{-1} (\l - \L_{\dual{\V}})^{-1} j'$ for every $\l\notin \Sigma_{\theta}$, and by the definition of the square roots via contour integrals, 
\[ 
\L_\H^{-\frac{1}{2}} = j'^{-1} \L_{\dual{\V}}^{-\frac{1}{2}} j'.
\]
(i)$\Rightarrow$(ii) Therefore, if $\cD(\L_\H^{\frac{1}{2}}) = \cR( \L_\H^{-\frac{1}{2}} ) = j(\V)$, then 
\begin{align*}
\cD(\L_{\dual{\V}}^{\frac{1}{2}}) & := \L_{\dual{\V}}^{-\frac{1}{2}} (\dual{\V}) = \L_{\dual{\V}}^{-\frac{1}{2}} \L_{\dual{\V}}(j'j (\V))\\
& =  \L_{\dual{\V}}^{-\frac{1}{2}} \L_{\dual{\V}}(j' \L_\H^{-\frac{1}{2}}(\H) ) =  \L_{\dual{\V}}^{-\frac{1}{2}} \L_{\dual{\V}}( \L_{\dual{\V}}^{-\frac{1}{2}} j'(\H) )\\
& = j'(\H),
\end{align*} 
where the last equality follows from $\L^{\frac{1}{2}}_{\dual{\V}} \L_{\dual{\V}}^{\frac{1}{2}}=\L_{\dual{\V}}$; see e.g. \cite[Theorem 15.15, p.289]{KuWe04}. By \cite[Corollary 2.3, p.113]{Ya10},  $j'(\H) = [\dual{\V}, \cD(\L_{\s, \dual{\V}})]_{\frac{1}{2}}$, where on $\cD(\L_{\s, \dual{\V}}) = J(\V)$ we consider the graph norm of $\L_{\s, \dual{\V}}$, that is, $\|\L_{\s, \dual{\V}} \cdot\|_{\dual{\V}} + \|\cdot\|_{\dual{\V}}$. 
Since for $v\in \V$ we have 
\[
I_\V \L_{\dual{\V}} Jv = (I+i\T )v \quad \textrm{ and }\quad I_\V \L_{\s, \dual{\V}} Jv = v,  
\]
hence 
\[ 
\|\L_{\dual{\V}} Jv \|_{\dual{\V}} = \|(I+i\T )v\|_\V \quad  \textrm{and }\quad \|\L_{\dual{\V}} Jv \|_{\dual{\V}} = \|v\|_\V. 
\] 
Consequently, the invertibility of $I+i\T$ implies, that the graph norm of $\L_{\s, \dual{\V}}$ is equivalent to the graph norm of $\L_{\dual{\V}} = \L_{\a, \dual{\V}}$ on $\cD(\L_{\dual{\V}}) = J(\V)$.
Therefore, we get that 
\[
 [\dual{\V}, \cD(\L_{\dual{\V}})]_{\frac{1}{2}} = \cD(\L_{\dual{\V}}^{\frac{1}{2}}).
\]
 Hence, by \cite[Theorem 16.3, p.532]{Ya10}, $\L_{\dual{\V}}$ has a bounded $H^\infty$ functional calculus.  
 
(ii)$\Rightarrow$(i) On the other hand,  if $\L_{\dual{\V}}$ has a bounded $H^\infty$ functional calculus, then as above we get $ \cD(\L_{\dual{\V}}^{\frac{1}{2}}) =
[\dual{\V}, \cD(\L_{\dual{\V}})]_{\frac{1}{2}} = j'(\H)$. 
Therefore, 
\begin{align*}
\cD(\L_{\H}^{\frac{1}{2}}) & := \L_{\H}^{-\frac{1}{2}} j'^{-1}(j'(\H)) = \L_{\H}^{-\frac{1}{2}} j'^{-1} \L_{\dual{\V}}^{-\frac{1}{2}} (\dual{\V})\\
& =  j'^{-1} j' \L_{\H}^{-\frac{1}{2}} j'^{-1} \L_{\dual{\V}}^{-\frac{1}{2}} (\dual{\V}) = j'^{-1} \L_{\dual{\V}}^{-\frac{1}{2}}  \L_{\dual{\V}}^{-\frac{1}{2}} (\dual{\V}) = j'^{-1} j'j (\V)\\
& = j(\V).
\end{align*}

For the last statement about the angle of the $H^\infty$ functional calculus first note, that $\cD(\L_{\dual{\V}}^{\frac{1}{2}}) = [ \dual{\V}, \cD(\L_{\dual{\V}})]_{\frac{1}{2}} = j'(\H)$ yields 
\[
L_H = j'^{-1} L_{\dual{\V}}^{-\frac{1}{2}} L_{\dual{\V}} L_{\dual{\V}}^{\frac{1}{2}} j'.
\] Moreover, by the Closed Graph Theorem, the operator $j'$ is an isomorphism from $H$ onto $\cD(\L_{\dual{\V}}^{\frac{1}{2}})=j'(H)$ equipped with the graph norm. Since the operator $L_H$ is $(\arctan \beta)$-accretive, therefore it is sectorial of angle $\arctan \beta$, and consequently the operator $L_{\dual{\V}}$, too. Finally, for example, by \cite[Theorem 16.3, p.532]{Ya10} (cf. \cite[Remark 16.2, p.536]{Ya10}), $L_{\dual{\V}}$ has a bounded $H^\infty$ functional calculus in any sectorial domain $\Sigma_\theta$ with $\theta > \arctan \beta$. This completes the proof. 
 \end{proof}

\begin{proof}[Proof of Theorem \ref{charact0}] Assume that $(\a, \H)$ has the Kato property. By  Lemma \ref{Arendt charact}, $\L_{\dual{\V}}$ has a bounded $H^\infty(\Sigma_\theta)$ functional calculus for every  $\theta > \arctan \beta$. Fix $\theta\in (\arctan \beta , \frac{\pi}{2})$. By the characterisation of the boundedness of the $H^\infty$ functional calculus,  \cite[Theorem 11.13, p.229]{KuWe04}, 
$\L_{ \dual{\V}}$ is $\theta$-accretive with respect to an equivalent inner product $\bracket{\cdot}{\cdot}_\theta$ on $\dual{\V}$.
Let $\widetilde \S \in \cL (\dual{\V} )$ be the positive operator such that $\bracket{\cdot}{\cdot}_\theta = \bracket{\widetilde{\S}\cdot}{\cdot}_{\dual{\V}}$. Then $\S := I_\V \widetilde{\S} I_\V^{-1} \in \cL (\V)$ is a positive operator on $\V$.

First, note that $I_\V \L_{\dual{\V}} J v = (I+i\T )v$ and $I_\V Jv = \Q v$ for every $v\in \V$. Then, 
\begin{align*}
\bracket{\L_{\dual{\V}}Jv}{Jv}_\theta & = \bracket{\widetilde \S \L_{\dual{\V}}Jv}{Jv}_{\dual{\V}}\\
& = \bracket{I_\V\widetilde \S I^{-1}_\V I_\V \L_{\dual{\V}} J^{-1}v}{I_\V J v}_\V\\
& = \bracket{I_\V  \L_{\dual{\V}} Jv}{\S I_\V J v}_\V\\
& = \bracket{I_\V J v}{\S(I+i\T )v}_\V\\
& = \bracket{(I+i\T )v}{\S\Q v}_\V\\
& = \bracket{\Q \S(I+i\T )v}{v}_\V
\end{align*} 
for every $v\in \V$. Therefore, the operator $\Q \S(I+i\T )$ is $\theta$-accretive with respect to $\bracket{\cdot}{\cdot}_\V$.
Therefore, (i)$\Rightarrow$(ii)$\Rightarrow$(ii$'$). The implication (ii$'$)$\Rightarrow$(i) follows from a similar argument.

The equivalences (ii)$\Leftrightarrow$(iii)$\Leftrightarrow$(iv) follow from the following chain of equivalences which holds for every positive operator $\S\in\cL (\V)$ and $\theta\in (0,\frac{\pi}{2}]$: 
\begin{align*}
 & \Q \S (I+i\T ) \text{ is }\theta\text{-accretive} \\
\Leftrightarrow \quad & \forall u\in \V : \, \bracket{\Q \S (I+i\T )u}{u}_\V \in \Sigma_\theta \\
\Leftrightarrow \quad & \forall u\in \V : \, \bracket{\Q \S u}{(I+i\T )^{-1}u}_\V  \in \Sigma_\theta \\
\Leftrightarrow \quad & \forall u\in \V : \, \bracket{u}{\S^\frac12 \Q (I+i\T )^{-1} \S^{-\frac12} u}_\V  \in \Sigma_\theta \\
\Leftrightarrow \quad & \S^\frac12  (I-i\T )^{-1}\Q \S^{-\frac12}  \text{ is }\theta\text{-accretive} \\
\Leftrightarrow \quad & \S^\frac12  \Q(I+i\T )^{-1} \S^{-\frac12}  \text{ is }\theta\text{-accretive}.
\end{align*}

For (iv)$\Leftrightarrow$(v), set $A:=\Q (I-i\T )^{-1}$, and note that its Cayley transform is given by  
\[
\cC := \phi(A) = (I-\Q+i\T ) (I+\Q +i\T )^{-1},
\]
where $\phi$ is the conformal map $\phi(z):= (1-z)(1+z)^{-1}$ from  $\Sigma_{\frac{\pi}{2}}$ onto  $\{|z|<1\}$. Moreover, for every polymomial $p$ we have
\[ 
p(\cC) = (p\circ \phi)(A) \quad \textrm{and } \quad \sup_{z\in \Sigma_\theta}|(p\circ \phi)(z)|\leq \sup_{|z|<1} |p(z)|.
\] 
 Therefore, the boundedness of the $H^\infty(\Sigma_\theta)$ functional calculus of $A$ with $\theta\leq \frac{\pi}{2}$ yields the polynomial boundedness of its Cayley transform $\cC$.
  For the converse, by Runge's theorem, it is easy to see that $A$ has a bounded $\mathcal{R}(\Sigma_{\frac\pi2})$ functional calculus; here $\mathcal{R}(\Sigma_{\frac\pi2})$ stands for the algebra of rational functions with poles outside $\Sigma_{\frac\pi2}$. 
Then, the boundedness of the $H^\infty(\Sigma_{\frac\pi2})$ functional calculus follows again by an approximation argument and McIntosh's convergence theorem \cite[Section 5, Theorem]{Mc86}; see also \cite[Proposition 3.13, p.66]{Hs06}. Since $A$ is $\theta$-sectorial for some $\theta<\frac\pi2$, \cite[Theorem 11.13]{KuWe04} gives (iv).

Finally, for (v)$\Rightarrow$(vi), since the Cayley transform $\cC= \phi(A)$ is a Ritt operator, see Lemma \ref{Ritt cond}, by \cite[Theorem 5.1]{LM98}, it is similar to a contraction. The converse is a consequence of the von Neumann inequality.
\end{proof}

\begin{remark}\label{rem1} 
(a) By \cite[Theorem 11.13 H7)]{KuWe04} one can show that if $(\a, \H)$ has the Kato property, then (iv) in Theorem \ref{charact0} holds with
\[
S: =  \int_{\Sigma_{\pi - \theta}} A^*e^{zA^*}Ae^{zA} \ud z  =\int_{\Sigma_{\pi - \theta}} |Ae^{zA}|^2 \ud z,
\]
where $A:=\Q (I-i\T)^{-1}$ and the integral exists in the weak operator topology.

(b) Note that in the case when the operator $\Q$ is invertible on $\V$, or equivalently, the inner products on $\H$ and $\V$ are equivalent, then $(\a, \H)$ has the Kato property simply because $\L_\H \in \cL(\H) = \cL(\V)$. 
It should be pointed out, that in this case the similarity to a contraction of the operator 
\begin{equation}\label{Fan for C}
\cC = (I-\Q + i\T )(I+\Q + i\T)^{-1} = (\T - i(I-\Q ))(\T - i(I+\Q ))^{-1},
\end{equation}
which is stated in Theorem \ref{charact0} (vi), can be proved in a straightforward way. Indeed, in \cite[Theorem 1]{Fa73} Fan proved that an operator $\cC\in \cL (\V)$ with $1\in \rho(\cC )$ is similar to a contraction if and only if it can be expressed in the form
\[ 
 (E - iF)(E - iG)^{-1},
\] 
for some selfadjoint operators $E$, $F$, $G \in \cL(\V)$ such that $G+F$ and $G-F$ are positive with  $0\in \rho(G-F)$. Therefore, in the case of $\Q$ being invertible, the above stated expression of the operator $\cC$, that is, \eqref{Fan for C}, satisfies these conditions.    
\end{remark}

\section{Kato property and triangular operators}

Recall that a bounded operator $\Delta$ on a Hilbert space $\V$ is {\it triangular} if there exists a constant $M\geq 0$ such that
\begin{equation}\label{deftrian}
 \left| \sum_{j=1}^{n} \sum_{k=1}^{j} \bracket{\Delta u_j}{v_k}_\V \right|
\leq M  \sup_{|a_j|=1}\left\| \sum_{j=1}^n a_j u_j\right\| \sup_{|a_j|=1}\left\| \sum_{j=1}^n a_j v_j\right\|.
\end{equation} for every $n\in \N$ and every
$u_1$, $\dots$, $u_n$, $v_1$, $\dots$, $v_n \in \V$. By a theorem of Kalton \cite[Theorem 5.5]{Ka07}, an operator $\Delta$ on $\V$ is triangular if and only if
$\sum_{n=1}^{\infty} \frac{s_n(\Delta )}{n+1}< \infty$, where $(s_n(\Delta ))_{n\in \N}$ is the sequence of singular values of $\Delta $. Therefore, the Schatten-von Neumann classes are included in the class of triangular operators.
We refer the reader to \cite[Section 5]{Ka07} for basic properties of triangular operators. One interest in the class of triangular operators stems from the following perturbation theorem by Kalton \cite[Theorem 7.7]{Ka07}.

\begin{lemma} \label{Kalton}
 Let $A$ and $\B$ be two sectorial operators on a Hilbert space $\H$. Assume that $\B$ has a bounded $H^\infty$ functional calculus, and that $A = (I+\Delta)\B$ for some triangular operator $\Delta$. Then $A$ has a bounded $H^\infty$ functional calculus, too. 
\end{lemma}

Combining this result with Theorem \ref{charact0}, we show that the Kato property of $(\a, \H)$ is preserved under certain {\it triangular perturbations} of the imaginary part of $\a$, and in particular, that for every bounded, selfadjoint operators $\T$ and $\Q$ on a Hilbert space $\V$ such that $\T$ is triangular and $\Q$ is nonnegative and injective, the pair $(\T , \Q )$ has the Kato property, that is, $\Q (I-i\T )^{-1}$ is similar to an accretive operator on $\V$.

\begin{corollary} \label{main} 
 Let $\a$ and $\b$ be two sectorial forms on $\V$ with the same real parts, that is, $\Re \a = \Re \b$. Let the imaginary parts $\t_\a$ and $\t_\b$ of $\a$ and $\b$ be determined by selfadjoint operators $\T_\a$, $\T_\b\in\cL (\V )$, respectively. Assume that $(\b, \H )$ has the Kato property, and that $\T_\a - \T_\b$ is a triangular operator. Then $(\a, \H )$ has the Kato property, too.

In particular, if $\T_\a$ is a triangular operator, then $(\a, \H )$ has the Kato property for every Hilbert space $\H$ into which $\V$ is densely and continuously embedded. 
\end{corollary}

\begin{proof}[Proof of Corollary \ref{main}] 
Note that by the second resolvent equation we get 
\[(I-i\T_\a)^{-1}\Q - (I-i\T_\b)^{-1}\Q = i(I-i\T_\a)^{-1}(\T_\a -\T_\b)(I+i\T_\b)^{-1}\Q .
\]
Therefore, since the operator $i(I-i\T_\a)^{-1}(\T_\a -\T_\b)$ is triangular, the claim follows from Lemma \ref{lem on sect}, Lemma \ref{Kalton}, and Theorem \ref{charact0} (iii).

Alternatively, note that Arendt's result, Lemma \ref{Arendt charact}, used in the proof of Theorem \ref{charact0}, can be directly applied to get Corollary \ref{main}. Indeed, set $\Delta := \L_{\a ,\dual{\V}} \L_{\b ,\dual{\V}}^{-1} - I,$ so that $\L_{\a ,\dual{\V}} = (I+\Delta )\L_{\b, \dual{\V} }$.  By our assumption, $\L_{\b, \dual{\V} }$ admits a bounded $H^\infty$ functional calculus. We also recall that both $\L_{\a, \dual{\V} }$ and $\L_{\b, \dual{\V} }$ are sectorial operators. By Lemma \ref{Kalton}, it is thus sufficient to show that the operator $\Delta$ is triangular. Since $\Re \a = \Re \b$, we get $\Delta = i \, (\L_{\Im\a, \dual{\V} } - \L_{\Im\b, \dual{\V} }) \, \L_{\b, \dual{\V} }^{-1}$. Fix $u$ and $v$ in $\dual{\V}$.
Then
\begin{align*}
 \bracket{\Delta  u}{v}_{\dual{\V}} & = \bracket{I_\V \Delta u}{I_\V v}_\V \\
& = i[(\L_{\Im\a, \dual{\V} } - \L_{ \Im\b, \dual{\V} })\L_{\b, \dual{\V}}^{-1}u](I_\V v)  \\
& = i \, \Im\a\bigl((\L_{\b ,\dual{\V}}J)^{-1} u , \, I_\V v\bigr) - i \, \Im\b\bigl((\L_{\b ,\dual{\V}}J)^{-1} u , \,I_\V v\bigr) \\
& = i \bracket{(\L_{\b ,\dual{\V}}J)^{-1} u }{(\T_\a - \T_\b) I_\V v}_\V.
\end{align*}
Since $\L_{\b ,\dual{\V}}J$ is an isomorphism from $\V$ onto $\dual{\V}$, the triangularity of $\Delta$ is equivalent to the triangularity of $\T_\a - \T_\b$.

For the proof of the second statement, it is sufficient to apply the one just proved for a symmetric form, that is, $\b$ with $\Im \b = 0$.   
\end{proof}

In an analoguous way, by combining Lemma \ref{Kalton} with Theorem \ref{charact0} (iii), we get the following perturbation result for the real parts of forms.

\begin{corollary}
Let $\a$ and $\b$ be two sectorial forms on a space $\V$ with the same imaginary parts, that is, $\t_\a = \t_\b$, and equivalent real parts $\s_a$ and $\s_b$. Let $\S\in \cL(\V)$ be such that $\s_\a (\S u,v) = \s_\b (u.v)$.

If $\S-I$ is triangular and  $(\b, \H )$ has the Kato property, then $(\a, \H )$ has the Kato property, too. 
\end{corollary}

\begin{proof}
According to Theorem \ref{charact0} (iii), if $(\H, \bracket{\cdot}{\cdot} )$ is a Hilbert space such that $(\b, \H)$ has the Kato property, then the operator $(I-i\T_\b)^{-1}\Q_\b$, where $\Q_\b$ is a nonnegative, injective operator on $(\V , \s_\b)$ with $\bracket{u}{v} = \s_\b (\Q_\b u, v)$ ($u$, $v\in \V$), has a bounded $H^\infty$-functional calculus. The corresponding operator $\Q_\a$ for the form $\a$ is equal to $\S\Q_\b$ and $\T_\a=\S\T_\b$. Hence, 
\[
(I-i\T_\a)^{-1}\Q_\a = (I - i\S\T_\b)^{-1} \S\Q_\b .  
\]
Then, since $(i\T_\b - i\S\T_\b)(I-i\T_\b)^{-1}\Q_\b $ is triangular and 
\[
 (I-i\S\T_\b)^{-1}\Q_\b - (I-i\T_\b)^{-1}\Q_\b = (I-i\S\T_\b)^{-1}(i\T_\b - i\S\T_\b)(I-i\T_\b)^{-1}\Q_\b 
\]  
the operator $(I-i\S\T_\b)^{-1}\Q_\b$ has a bounded $H^\infty$ functional calculus. 
Moreover, note that 
\begin{align*}
(I-i\S\T_\b)^{-1}\S\Q_\b - (I-i\S\T_\b)^{-1}\Q_\b & = (I-i\S\T_\b)^{-1}(\S- I) \Q_\b\\
& = \Delta (I-i\S\T_\b)^{-1}\Q_\b, 
\end{align*}
where 
\[
 \Delta = (I-i\S\T_\b)^{-1}(\S- I) (I-i\S\T_\b)
\] 
is triangular. Therefore, again by Lemma \ref{Kalton}, $(I - i\S\T_\b)^{-1} \S\Q_\b$ has a bounded $H^\infty$ functional calculus, which completes the proof. 
\end{proof}

Finally, for the sake of completeness, we state a perturbation result for the operator $\Q$ generating the Hilbert space $\H$ in $(\a, \H)$. Its proof follows directly from Lemma \ref{Kalton} and Theorem \ref{charact0}(iv).

\begin{corollary}\label{cor 3} 
Assume that $(\a, \H )$ has the Kato property. Let $\Q$ be the nonnegative, injective operator on $\V$ associated with $\H$. Then, $(\a, \H_{\hat{\Q}} )$ has the Kato property for every Hilbert space $\H_{\hat{\Q}}$ with  $\hat{\Q}$ being a triangular perturbation of $\Q$, that is, $\hat{\Q} = (I+\Delta)\Q$ for some triangular operator $\Delta$. 
\end{corollary}

\section{Optimality of Corollary \ref{main}}
Below, we show that the class of triangular operators is, in a sense, the largest subclass of compact operators for which Corollary \ref{main} holds. Recall that, by \cite[Theorem 5.5]{Ka07}, a compact operator $\T$ is not triangular, if $\sum_{n\in \N}\frac{s_n(\T)}{n} = \infty$, where $s_n(\T)$ (${n\in \N}$) stands for the $n$-th singular value of $\T$.

\begin{proposition}\label{counterex}
Let $(a_n)_{{n\in \N}}$ be a nonincreasing sequence of positive numbers with $\sum_{{n\in \N}}\frac{a_n}{n}=\infty$. 
Then, there exists a sectorial form $\a$ such that the singular values $(s_n(\T_\a))_{{n\in \N}}$ of the operator $\T_\a$ determined by the imaginary part of $\a$ satisfy $s_n(\T ) \preceq a_n$ $({n\in \N})$, but not for every Hilbert space $\H$ for which  $\V$ is densely and continuously embedded in $\H$, the couple $(\a, \H)$ has the Kato property.   

Equivalently, there exist a selfadjoint, compact operator $\T$ on a Hilbert space $\V$ with $s_n(\T )\preceq a_n$  $({n\in \N})$, and a nonnegative, injective operator $\Q$ on $\V$ such that $\Q(I+i\T )^{-1}$ is not similar to an accretive operator.
\end{proposition}

In order to construct an example we adapt two related results from \cite{Ka07} and \cite{AuMcNa97I}. 
Recall that, in \cite{AuMcNa97I}, the sesquilinear form $\a$ on a Hilbert space $\H$ is expressed as
\begin{equation}\label{repr1}
\a(u,v) = \bracket{A\S u}{\S v}_\H , \quad u,v\in \V:=\cD(\S),  
\end{equation}
where $\S$ is a positive selfadjoint (not neccessarily bounded) operator on $\H$, and $A$ is a bounded invertible $\theta$-accretive operator on $\H $ for some $\theta<\pi/2$. (Here, we call the selfadjoint operator $\S$ on $\H$ \emph{positive} if $\bracket{\S u}{u}_\H >0$ for all $u\in \cD(A)\setminus\{0\}$.) Then, following Kato's terminology \cite{Kat61}, $\a$ is a {\it regular accretive form} in $\H$.
The operator $\L_{\a,\H}$ associated with the form $\a$ on $\H$ is given by $\S A\S$. Note that $\s = \Re \a$ is an equivalent inner product to $\bracket{\S\cdot}{\S\cdot}_\H$, and in order to put it in our setting, we additionally assume that $0\in \rho(\S)$. Then $\s$ is a \emph{complete} inner product on $\V:=\cD(\S)$. In fact, since $\S$ is selfadjoint, to get the completeness of this inner product, it is sufficient that $\S$ is injective and has closed range.   

For the convenience of the reader we restate two auxiliary results from \cite{Ka07} and \cite{AuMcNa97I}.

\begin{lemma}[{\cite{Ka07}, Theorem 8.3}] \label{Kalton result}
Let $\H $ be a separable Hilbert space and let $(e_n)_{{n\in \N}}$ be
an orthonormal basis. Let $\S$ be the sectorial operator defined by
$\S e_n=2^ne_n$ $({n\in \N})$ with $\cD(\S):=\{x\in \H : \sum_{n=1}^\infty 2^{2n}|\bracket{x}{
e_n}_\H |^2 <\infty\}$. Suppose $K\in \cL(\H )$ is a non-triangular
compact operator. 
Then, there exist bounded operators $U$ and $W$ on $\H$ such that for every
$m\in \N$, the operator $(I+2^{-m}WKU)\S$ fails to have a bounded $H^\infty$ functional calculus.
\end{lemma}

\begin{lemma}[{\cite{AuMcNa97I}, Theorem 10.1}] \label{AMN result} 
Let $A$, $\S$, $\a$ have the properties specified above. Then $(\a, \H )$ has the Kato property if and only if the operator $A\S$ has a bounded $H^\infty$ functional calculus.
\end{lemma}

\begin{lemma}\label{singular values} 
Let $A$, $\S$, $\a$ have the properties specified above. Let $\T$ and $\Q$ be the operators associated with $\a$ and $\H$. 
\begin{itemize}
 \item [(i)] The operator $\T\in \cL(\V)$ is compact if and only if the operator $\Im A \in \cL(\H)$ is compact. Then, 
$s_n(\T)\simeq s_n(\Im A)$ ($n\in \N$), that is,  there exists $c>0$ such that $c^{-1} s_n(\T)\leq s_n(\Im A) \leq c s_n(\T)$ for every $n\in \N$. 
 \item [(ii)] The operator $\Q  \in \cL (\V)$ is compact if and only if the embedding of $\V$ into $\H$ is compact, if and only if $\S^{-1}\in \cL(\H)$ is compact. Then,  $s_n(\Q) \simeq s_n(\S^{-1}) \simeq s_n(j)$ ($n\in \N$) where $j$ denotes the canonical embedding of $\V$ into $\H$. 
\end{itemize}
\end{lemma}

\begin{proof} First, note that the operators $\T$ and $\Q$ are of the form:
\begin{align*}
 \T & = \S^{-1}(\Re A)^{-1}\Im A\, \S\quad \text{and} \\
 \Q & = \S^{-1}(\Re A)^{-1} \S^{-1}_| ,
\end{align*}
where $\Re A$ and $\Im A$ denote the real and the imaginary part of $A$, and $\S^{-1}_|$ is the restriction of $\S^{-1}\in \cL(\H)$ to $\V$, considered as an operator in $\cL (\V ,\H)$. These expressions give the first statements in (i) and (ii). The second assertion in (i) follows in a straightforward from, e.g., \cite[Theorem 7.7, p. 171]{Wei80}. \\

To prove the corresponding one of (ii), assume that $\S^{-1}$ is compact with spectrum $\sigma(\S^{-1})=:\{\mu_n\}_{n\in\N}$, where $\mu_n \rightarrow 0^+$ as $n\rightarrow \infty$. Therefore, there exists an orthonormal system $\{e_n\}_{n\in\N}$ in $\H$ such that $\S h =\sum_{n} \mu_n^{-1} \bracket{h}{e_n}_{\H} e_n$ for $h\in \cD(\S) = \{h\in \H: \sum_n \mu_n^{-2} |(h,e_n)|^2 < \infty\}$. 
Let $\cC :\V_* \rightarrow \H$, $\cC u:= \S^{-1}u$, $u \in \cD(\S)$, where 
$\V_*$ denote the Hilbert space $(\cD(\S), \bracket{\S\cdot}{\S\cdot}_\H)$. Of course, $\cC\in \cL(\V_*,\H)$ and $\cC^*\cC\in \cL(\V_\star)$ are compact. Moreover, note that 
\[
 \cC^*\cC u=\sum_n \mu_n^2 \bracket{u}{g_n}_{\V_\star} g_n,\quad \quad u\in \V_\star,
\]
where $g_n:=\mu_n^{-2}e_n$ $(n\in \N)$ is an orthonormal basis for $\V_*$.  
Thus, the singular values of $\cC$ are given by $s_n(\cC)=\mu_n$, $n\in \N$.

Now, let $I_\S$ denote the identity map on $\cD(\S)$ considered as an operator from $\V$ onto 
$\V_*$. Therefore, we have  $\S^{-1}_| = \cC I_\S$ and, by \cite[Theorem 7.1, p. 171]{Wei80}, we get 
\begin{align*}
s_n(\Q) &\leq \|\S^{-1} (\Re A)^{-1}\|_{\cL(\H)} s_n(\cC I_\S) = \|\S^{-1} (\Re A)^{-1}\|_{\cL(\H)} s_n(I_\S^* \cC^*)\\
  &\leq  \|\S^{-1} (\Re A)^{-1}\|_{\cL(\H)} \|I_\S^*\|_{\cL(\V_\star, \V)}  s_n(\cC^*) \\
  &\leq  \|\S^{-1} (\Re A)^{-1}\|_{\cL(\H)} \|I_\S^*\|_{\cL(\V_\star, \V)}  s_n(\cC) \quad \textrm{ and }\\
s_n(\cC) & = s_n( (\Re A) \S \Q I_\S^{-1}) \\
  & \leq  \|(\Re A) \S\|_{\cL(\V ,\H)} s_n(\Q I_\S^{-1})\\
  & \leq \|(\Re A) \S\|_{\cL(\V ,\H)} \|I_\S^{-1}\|_{\cL(\V_\star ,\V)} s_n(\Q).
\end{align*}
Finally, note that $s_n(\S^{-1})$ is equal to the $n$-th singular value of the embedding of $\V_*$ into $\H$. This completes the proof. 
\end{proof}

\begin{proof}[Proof of Proposition \ref{counterex}]
Suppose that $\H $, $\S$, $K$, $U$, $W$ have the properties specified above in
Lemma \ref{Kalton result}. Fix $m\in \N$ such that the numerical range of the operator $A:=I+2^{-m}WKU)$ is contained in $\{|z-1|<1\}$. Then, by Lemma \ref{AMN result}, the couple $(\a, \H )$ does not have the Kato property, where the form $\a$ corresponds to the operators $A$ and $\S$ as above.  
Moreover, by Lemma \ref{singular values}(i), 
\[ 
s_n(\T)\simeq s_n(\Im A) \simeq s_n\bigl(2^{-m}(WKU-U^*K^*W^*)\bigr)\preceq s_n(K) \quad (n\in \N).
\]
This completes the proof. 
\end{proof}

\nocite{Kat62}

\providecommand{\bysame}{\leavevmode\hbox to3em{\hrulefill}\thinspace}

\bibliographystyle{abbrv}
\bibliographystyle{amsplain}

\end{document}